\documentclass[12pt,twoside]{irmaems} 
\usepackage[english]{babel}
\usepackage{amsmath,amssymb}
\usepackage{float,enumerate}

  \ifx\pdftexversion\undefined
    \usepackage[dvips]{graphicx,color}
  \DeclareGraphicsExtensions{.jpg,.eps,.pnm}
  \else
   \usepackage[pdftex]{graphicx,color}
 \fi

\let\ifanglais\iftrue

\def\R{{\mathbb R}}

\def\Z{{\mathbb Z}}

\renewcommand{\leq}{\leqslant}
\renewcommand{\geq}{\geqslant}

\renewcommand\theequation{\arabic{section}.\arabic{equation}}

\theoremstyle{definition} %%% for statements in roman typeface

 \newtheorem{defi}{Definition}[section]

 \newtheorem{ex}[defi]{Example}

\theoremstyle{plain}      %%% for statements in italic typeface

 \newtheorem{prop}[defi]{Proposition}
 \newtheorem{theo}[defi]{Theorem}

\begin{document}
\title{On the Hilbert Geometry of Convex Polytopes}
\author{{\sl Constantin Vernicos}}

\address{%
  \sl Institut de mathématique et de modélisation de Montpellier\\
  \sl Université Montpellier 2 \\
  \sl Case Courrier 051\\
  \sl Place Eugène Bataillon \\
  \sl F--34395 Montpellier Cedex\\
  \sl France\\
\sl email:\, \tt{Constantin.Vernicos@um2.fr}}
\maketitle

\begin{abstract}
We survey the Hilbert geometry of convex polytopes. In particular we present two important characterisations of 
these geometries, the first one
 in terms of the volume growth of their metric balls, the second one  as a bi-lipschitz class of the
simplexe's geometry.
\end{abstract}

\tableofcontents

\section{Introduction}

Our understanding of the Hilbert geometry associated to the interior of a convex polytope has increased tremendeously 
in the last decade. 
Polytopes play an important role in the realm of Hilbert geometries because they are related to
 linear programming and their geometry is somehow more amenable and simple
than that of general convex sets.

This chapter aims at presenting various results and characterisations of the Hilbert geometry of polytopes. 
For instance their group of isometries is now well understood (see Section \ref{isometries}).
Their volume growth is polynomial and its order characterises them (Section \ref{volumegrowth}).
They are the only Hilbert geometry which can be isometrically embedded in a normed vector space (see Section \ref{embeddings}).
For a given dimension, they all belong to the same bi-Lipschitz class with the Euclidean metric space (see Section \ref{bilipschitz}).

When it was enlightening we dared offer our own proofs on some of the results presented here. Hence one will find
an outline of a new proof that in dimension $2$ the unique Hilbert geometry isometric to a normed vector space is
the one associated to a triangle. 
We also show that any polytope with $N+1$ faces can be isometrically embedded in a normed vector
space of dimension $N$. This last results improves the dimension of the target space which was previously known to be $N(N+1)/2)$
\cite{bclin}, and is more geometric in nature than the previous proofs. 
We end this chapter by outlining a proof of the fact that the Hilbert geometry of a polytope is bi-Lipschitz equivalent to a normed vector space of the same dimension.

\section{Definitions related to Hilbert Geometry, polytopes and notation}

\subsection{Hilbert Geometries}

Let us recall that a Hilbert geometry
$(\mathcal{C},d_\mathcal{C})$ is a non-empty bounded open convex set $\mathcal{C}$
on $\R^n$, that we shall call \textsl{convex domain}, with
the Hilbert distance 
$d_\mathcal{C}$ defined as follows : for any distinct points $p$ and $q$ in $\mathcal{C}$,
the line passing through $p$ and $q$ meets the boundary $\partial \mathcal{C}$ of $\mathcal{C}$
at two points $a$ and $b$, such that someone walking on the line goes consecutively by $a$, $p$, $q$,
$b$ (Figure~\ref{dintro}). Then we define
$$
d_{\mathcal C}(p,q) = \frac{1}{2} \log [a,p,q,b],
$$
where $[a,p,q,b]$ is the cross-ratio of $(a,p,q,b)$, i.e., 
$$
[a,p,q,b] = \frac{\| q-a \|}{\| p-a \|} × \frac{\| p-b \|}{\| q-b\|} > 1,
$$
and $\| \cdot \|$ is the standard Euclidean norm in
$\mathbb R^n$.
   
\begin{figure}[h] 
  \centering
  \includegraphics[scale=.4]{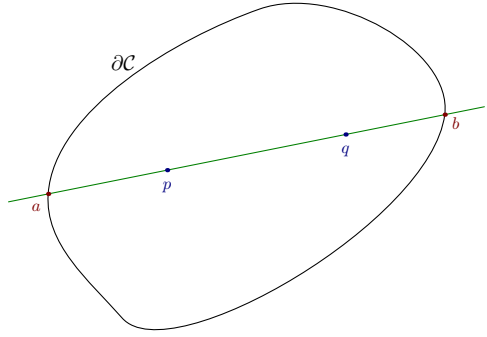}
  \caption{The Hilbert distance \label{dintro}}
\end{figure}

Note that the invariance of cross-ratios under projective maps implies the invariance 
of $d_{\mathcal C}$ by such maps.

These geometries are naturally endowed with
a  continuous Finsler metric $F_\mathcal{C}$ defined as follows: 
if $p \in \mathcal C$ and $v \in T_{p}\mathcal C =\R^n$
with $v \neq 0$, the straight line passing by $p$ and directed by 
$v$ meets $\partial \mathcal C$ at two points $p^{+}$ and
$p^{-}$~; we then define
\begin{equation}
  \label{eqhilbertisfinsler}
  F_{\mathcal C}(p,v) = \frac{1}{2} \| v \| \biggl(\frac{1}{\| p -
  p^{-} \|} + \frac{1}{\| p - p^{+}
  \|}\biggr) \quad \textrm{and} \quad F_{\mathcal C}(p, 0) = 0.
\end{equation}

\begin{figure}[h]
  \centering 
    \includegraphics[scale=.4]{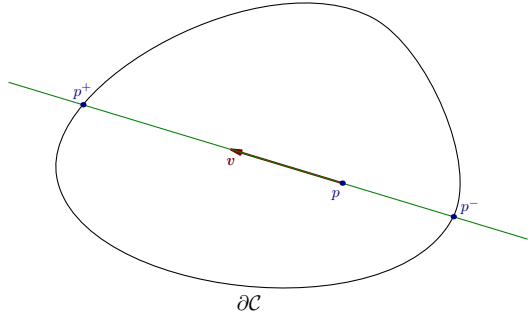}
          \caption{The Finsler structure \label{finslerintro}}
          \end{figure}

The Hilbert distance $d_\mathcal{C}$ is the length distance associated to the Finsler metric $F_{\mathcal C}$.

\subsection{Faces}\label{faces}

To an arbitrary closed convex set $K$ of a real vector space we can associate an equivalence relation, stating
that two points $A$ and $B$ are equivalent if there exists a segment $[C,D]\subset K$ containing
the segment $[A,B]$ such that $C\neq A,B$ and $D\neq A,B$. The equivalence classes are
called \textsl{faces}. A face is called a \textsl{$k$-face} when the dimension
of the affine space it generates is $k$. Observe that a face is always open in the affine space it
generates.
As usual we call \textsl{vertex} a \mbox{$0$-dimensional} face.

In this chapter a \textsl{simplex} in $\R^n$ is the convex closure of $n+1$ affinely independent points, that is a
triangle in $\R^2$, a tetrahedron in $\R^3$, etc. 
More generally, in this survey, a \textsl{polytope} in $\R^n$ will be the convex hull 
of a finite number of points, such that
$n+1$ of them are affinely independent. The $n$-face of a polytope is its interior, in particular it is never empty.

The next definition is due to Benzecri~\cite{benzecri} and plays an important role 
in the study of convex sets.

\begin{defi}[Conical faces]
  Let $\mathcal{C}$ be a convex set in $\R^n$. Let $k<n$.
Suppose that a simplex $S$ contains $\mathcal{C}$ and
that a non-empty $k$-face $f\subset\partial\mathcal{C}$, is included in a $k$-face of $S$. Then we
say that $f$ is a \textsl{conical face} of $\mathcal{C}$ and that $\mathcal{C}$ admits a conical face.
\end{defi}

\begin{figure}[H]
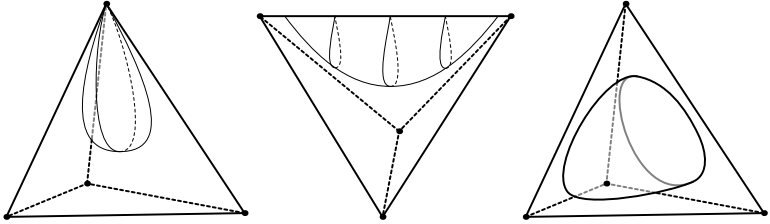

  \centering
  \includegraphics[scale=.37]{polygones.15}  
\includegraphics[scale=.37]{polygones.16}
\includegraphics[scale=.37]{polygones.17}

  \caption{Conical faces in dimension 3}
\label{conicalface}
\end{figure}

When a face $f$ is contained in the boundary of another 
face $F$ we write $f<F$.

\begin{defi}[Conical flag]
Let $\mathcal{C}$ be a convex set in $\R^n$. If there exists a simplex $\mathcal{S}$ contained in $\mathcal{C}$,
and a sequence of faces $(f_i)_{0\leq i\leq n-1}$ such that for any $k=0,\ldots,n-1$,
\begin{enumerate}
\item $\emptyset<f_0<f_1<f_2<\ldots<f_{n-1}<\mathcal{S}$,
\item $f_k$ is a subset of a $k$-conical face of $\mathcal{C}$;
\item no other $k$-face of $\mathcal{S}$ is in the interior of a $k$-conical face of $\mathcal{C}$;
\end{enumerate}
then we  call $f_0<f_1<f_2<\ldots<f_{n-1}<\mathcal{C}$ a \textsl{conical flag} and we say that  $\mathcal{C}$ admits
a \textsl{conical flag}.
Furthermore we will call $\mathcal{S}$ a \textsl{conical flag neighborhood} of $\mathcal{C}$.
\end{defi}

\section{Alternate viewpoints on the Hilbert metric of Polytopes}

We present here two explicit formulas arising from two different viewpoints on the Hilbert metric. The first one is due to Garett~Birkhoff
and the second one to Ralph~Alexander. Both are useful in some applications (see Section \ref{embeddings}).

\begin{prop}[G.~Birkhoff~\cite{birkhoff}]\label{eqbirkhoff}
  \sl Consider a convex polytope ${\mathcal P}$ in $R^n$
defined by the $N$ affine maps $L_1,\ldots,L_N\colon \R^n\to \R$ as follows 
$$
{\mathcal P}=\{x\in \R^n\mid L_i(x)>0, 1\leq i\leq N\}\text{.}
$$ 
Then for any pair of points $(x,y)$ contained in the interior of ${\mathcal P}$ one has
\begin{equation*}
  d_{{\mathcal P}}(x,y)=\frac12 \sup_{1\leq i,j\leq N} \log \biggl( \dfrac{L_i(x)}{L_i(y)}× \dfrac{L_j(y)}{L_j(x)} \biggr)\text{.}
\end{equation*}
\end{prop}

This formula is a consequence of the  convexity and 
the fact that if $z_i$ is the intersection of the straight line $(xy)$ with the hyperplane $H_i=\{L_i=0\}$,
then by Thales's theorem we have
$$
\frac{||x-z_i||}{||y-z_i||}= \dfrac{L_i(x)}{L_i(y)}\text{.}
$$ 

In dimension two a second point of view is available: focus on the angle defined by two lines $H_1$ and $H_2$ 
intersecting at a point 
$p$. Consider now two points $x$ and $y$ lying the interior of the same sector $Q$ defined by these two lines
and let $X^*$ and $Y^*$ be the straight lines $(px)$ and $(py)$ respectively.
Assuming that the four lines $H_1$, $X^*$, $Y^*$ and $H_2$ appear in that order
let $h_Q(x,y)$ be their cross-ratio, and define
$$\delta_Q(x,y)=\frac12 \log h_Q(x,y)\text{.}$$
In other words, $\delta_Q$ is the cone-metric associated to the cone $Q$. It is equal to $0$ if $(px)$ and $(py)$ define
the same line.
If $H_i$ is given by the equation $L_i=0$ for $i=1,2$ then we have

\begin{equation}
  \label{eqcrofton}
  \delta_Q(x,y)=\frac12 \Biggr| \log \biggl| \dfrac{L_1(x)}{L_1(y)}× \dfrac{L_2(y)}{L_2(x)} \biggr| \Biggr|\text{.}
\end{equation}

From this last remark we can now state the following proposition, which is a kind of \textsl{Crofton formula}
related to the Hilbert geometry of plane polytopes, \textsl{i.e.}, polygons. 
This means that we relate the length of a segment to a measure on the set of lines crossing 
that segment, see \cite{alvarez_fernandez}. 
 
\begin{prop}[R.~Alexander~\cite{alexander}]\label{alexander}
\sl  Let ${\mathcal P}$ be a polygone with vertices $p_1,\ldots,p_N$, and for $i=1,\ldots,N$ let $Q_i$ be the interior angle defined
at the vertex $p_i$; then the Hilbert distance in ${\mathcal P}$ is given by
$$
d_{{\mathcal P}}(x,y)=\frac12 \sum_{i=1}^n \delta_{Q_i}(x,y)\text{.}
$$
\end{prop}

To obtain this formula, it suffices to sum up the $\delta_{Q_i}$'s corresponding to the vertices lying on one side of
the straight line $(xy)$, which gives $d_{{\mathcal P}}(x,y)$.
A somewhat similar description in higher dimensions has been given by Rolf~Schneider~\cite{schneider}.

\section{The group of isometries versus collineations for a polytopal Hilbert geometry}\label{isometries}

Let $\sigma_{n+1}$ be the group of permutations on the set $\{1,\ldots,n+1\}$ and
let us denote by $\Gamma_{n+1}=\sigma_{n+1}×\sigma_2=\sigma_{n+1}×\Z/2\Z$.

Let us remind the reader that a \textsl{collineation}  is a  bijection from one 
projective space to another, or from a projective space to itself, such that the images 
of collinear points are themselves collinear.
A \textsl{homography} is an isomorphism of projective spaces, 
induced by an isomorphism of the vector spaces from which they are derived.
Homographie are collineations but in general not all collineations are homographies.
However the \textsl{fundamental theorem of projective geometry} asserts 
that in the case of \textsl{real} projective spaces of dimension at least two a collineation is a homography.

\begin{theo}
\sl  Let $({\mathcal P},d_{\mathcal P})$ be a polytopal Hilbert geometry of dimension $n$.
Then the group of isometries is isomorphic to 
\begin{enumerate}
\item $\R^n \rtimes \Gamma_{n+1}$ if ${\mathcal P}$ is a simplex;
\item the group of collineations of ${\mathcal P}$ otherwise.
\end{enumerate}
\end{theo}

The first part of this theorem was known, see for instance Pierre~de la 
Harpe~\cite{dlharpe}. The second part of the 
theorem is due to Bas Lemmens and Cormac Walsh~\cite{lw} following ideas of
 Cormac~Walsh~\cite{walsh} (see also
his contribution to this handbook \cite{walsh2}).

Their main idea is to study the  Horoboundary of a Polytopal Hilbert geometry. They show that
one can define a metric, the detour metric, between Busemann points which extends somehow the Hilbert metric.
For this metric the Busemann boundary is divided into different parts. Two points between different parts
are at an infinite distance, in particular Busemann points related to two different faces of the polytope
are not in the same part.
Then their strategy consists in proving the following facts: Given  an isometry $f$ between two polytopal Hilbert
geometries:
\begin{enumerate}[(i)]
\item The map $f$ defines an isometry between their respective Busemann boundaries
endowed with the detour metric;
\item either $f$ maps vertex parts to vertex parts and faces to faces or $f$ interchanges them;
\item if $f$ maps vertex parts to vertex parts, then $f$ extends continuously to the boundary;
\item if $f$ extends continuously to the boundary, then it is a collineation;
\item if $f$ interchanges vertex parts and faces, then the two polytopes are simplices.
\end{enumerate}

Notice that although it is now known that  for a general 
Hilbert geometry the group of isometries is a Lie group (see  L.~Marquis' contribution to this handbook~\cite{marquis}), 
it is still not known when it coincides with its group of collineations.

\section{Characterisation by volume growth}\label{volumegrowth}

\begin{theo}
\sl  Let $(\Omega,d_{\Omega})$ be a Hilbert geometry of dimension $n$. Let  $\mathrm{Vol}$ be its Hausdorff or Holmes-Thompson measure and let $B_\Omega(o,R)$ be the metric ball of radius $R$ centerd at $o$. Then the upper asymptotic volume, which is defined as
$$
\overline{\mathrm{Asvol}}\bigl(\Omega\bigr)=\limsup_{R\to +\infty} \dfrac{\mathrm{Vol}\bigl(B_\Omega(o,R)\bigr)}{R^n}\text{,}
$$ 
is finite if and only if $\Omega$ is a polytope.
\end{theo}

The fact that the upper asymptotic volume of a polytopal Hilbert geometry is finite
was proved by the author in \cite{ver9}.
The converse is also due to the author and a complete proof can be found in \cite{ver10}.
More precisely, in that paper we prove the following lower bound on the asymptotic volume:

\begin{theo}\label{lowerboundasymptoticvolume}
\sl There exists a constant $a_n$ such that for
any Hilbert geometry $(\Omega,d_\Omega)$ of dimension $n$ which 
admits $k$ extremal points one has
$$
a_n\cdot k\leq \overline{\mathrm{Asvol}}\bigl(\Omega\bigr)\text{.}
$$
\end{theo}
Hence having finite asymptotic volume implies that the convex set $\Omega$ has a finite number of 
extremal points and therefore is a polytope.

The proof of Theorem \ref{lowerboundasymptoticvolume} relies on the identity (\ref{eqlimit}) below
and on the fact that the measure of balls of radius $R$ have a uniform lower bound
of the form $b_nR^n$, for some constant $b_n$ depending only on the dimension (see \cite{ver10}).
Then it suffices to include in a ball of radius $R$, $k$-disjoint balls of radius $R/4$, centered on the 
geodesic rays joining the center of the ball to the extremal points. This is possible precisely thanks to the
the equality (\ref{eqlimit}).

Notice that finding a similar upper bound as in Theorem \ref{lowerboundasymptoticvolume} is not yet done, and
would probably solve the entropy upper-bound conjecture using the methods developed by the author
 in \cite{ver11}
to prove it in dimensions $2$ and $3$.

\section{Characterisation by isometric embedding}\label{embeddings}

\subsection{The special case of simplices}

\begin{theo}\label{lecasdessimplexes}
\sl  Let $(\Omega,d_\Omega)$ be a Hilbert geometry. It is isometric to a normed
vector space if and only if $\Omega$ is projectively equivalent to a simplex.
\end{theo}

The ``if'' part was proved by Roger Nussbaum~\cite{nussbaum} and Pierre~de la~Harpe~\cite{dlharpe}. 
The ``only if'' is due to Thomas Foertsh and Anders Karlsson \cite{fk}.

Let us illustrate the two-dimensional case with an \textsl{ad hoc} proof not requiring the technicality
of \cite{fk}.
Let $e_1$, $e_2$ and $e_3$ be an affine basis of an affine plane;
then the convex hull of these three points is a two-dimensional simplex $\mathcal{S}_2$.

Now using the barycentric coordinates attached to the family $(e_i)_{1\leq i\leq3}$, 
each point $p$ in the interior of the simplex 
is uniquely associated to a triple of positive real numbers $\alpha_1,\alpha_2,\alpha_3$ such that $\sum_i\alpha_i=1$ and
$p=\sum_i\alpha_i e_i$.
Therefore, one can define a map from $\mathcal{S}_2$ to 
the plane $\bigl\{x+y+z=0\bigr\}$ of $\R^3$  by
\begin{equation}
  \label{eqformultriangle}
  \Phi_2(p)=\biggl(\log \frac{\alpha_1}{\alpha_2}, \log \frac{\alpha_2}{\alpha_3}, \log \frac{\alpha_3}{\alpha_1}\biggr)\text{.}
\end{equation}

This map is easily seen to be a bijection whose inverse is
\begin{equation}
  \label{eqinverseformuletriangle}
  \Phi_2^{-1}(x,y,z)= \frac{1}{e^x+e^{x+y}+1}(e^{x+y},e^y,1)\text{.}
\end{equation}

Finally, if $\R^3$ is endowed with the sup norm, then this map is an isometry.

Now, the intersection of the unit cube of $\R^3$ with the plane $x+y+z=0$ is
a regular hexagon, and therefore we deduce from this that the Hilbert geometry
of a simplex is isometric to $\R^2$ endowed with a norm whose unit ball is a regular hexagon.

Conversely, and without loss of generality, suppose that $\Omega$ is a planar  
bounded convex set whose Hilbert geometry is isometric to a two-dimensional normed vector space. 
Since its volume growth being polynomial
of order two,  it follows, from \cite{ver10}, that 
$\Omega$ is necessarily a polygon. Besides, in dimension $2$, the length of a sphere
of radius $R$ in a normed vector space is $c× R$ for a constant $6\leq c\leq8$. 
However in a polygon with $n$
vertices a simple computation shows that as $R$ goes to infinity, the length of a sphere of
radius $R$ is equivalent to $2n× R$. Hence, $n=3$ or $4$. 
The case $n=4$ would mean that $\Omega$ is a convex quadrilateral, which is projectively equivalent to
a square; however the square is not isometric to a normed vector space, as in the center the finsler norm is a square, 
on the diagonals an hexagon and elsewhere an octagon.
Therefore, $n=3$ and $\Omega$ is a triangle.

\subsection{Isometric embeddings of polytopes}
\begin{theo}\label{plongementisometrique}
\sl For a  convex domain $\Omega \subset \mathbb{R}^n$,
the following conditions are equivalent:
\begin{enumerate}[(a)]
  \item $\Omega$ is a bounded polytope. 
  \item The Hilbert geometry $(\Omega, d_{\Omega})$  can be isometrically embedded  in a
  finite dimensional normed vector space 
\end{enumerate}
\end{theo}

Observe that condition (b) means that there exists a norm $\|\cdot \|$ on $\mathbb{R}^m$ for some $m \in \mathbb{N}$, 
 and a map $f : (\Omega, d_{\Omega}) \to (\mathbb{R}^m, \| \, \|)$ which is an isometry onto its image. 
By Theorem , the image $f (\Omega) \subset \mathbb{R}^m$ is an affine subspace if and
 only if $\Omega$ is a simplex. 

The implication  $(a) \Rightarrow (b)$ is due to Brian Lins who proved it in his dissertation \cite{bclin}.
He used Birkhoff's result Proposition \ref{eqbirkhoff} and obtained  an embedding $f : \Omega \to \mathbb{R}^{N(N+1)/2}$ for the sup norm, where $N$ is one less the number of faces of the polytope. 

A more geometric proof of this implication is to see it as an immediate consequence of Theorem 
\ref{lecasdessimplexes},  
together with
the following result  which states that any bounded polytope is affinely  equivalent to the intersection 
of a simplex 
and an affine subspaces in some vector space  (see also  \cite{grunbaum} Theorem 1 in section 5.1). 
The argument also
reduces the dimension of the ambient space from $N(N+1)/2$ to $N$.

\begin{prop}
\sl  Let ${\mathcal P}$ be a convex and bounded polytope in $\R^n$ with 
$N+1$ $(n-1)$-faces. Then there exists an $N$-simplex $S_{N}$ and an
$n$-dimensional affine space $A_n$ in $\R^N$ such that ${\mathcal P}$ is affinely equivalent to $A_n\cap S_{N}$.
\end{prop}
\begin{proof}
Let $L_i\colon \R^n\to \R$ be affine function, for $1\leq i\leq N+1$ such that
$$
{\mathcal P}=\{x\in \R^n\mid L_i(x)>0, 1\leq i\leq N+1\}\text{.} 
$$
Notice that necessarily $N\geq n$ for the convex to
be bounded. Let the family $(e_i)_{1\leq i\leq n+1}$ be an affine basis of $\R^n$ and let us
suppose that in that basis, one has,  in barycentric coordinates $x=\sum_i x_ie_i$ and $\sum_i x_i=1$, 
$$
L_i(x)=a_1^ix_1+\cdots+a_{n+1}^ix_{n+1}\text{,}
$$
where for each $1\leq i\leq n+1$ the $a_j^i$ are not all equal and,
without loss of generality, we can suppose that the first $n+1$ hyperplanes are affinely independent.

Now let us consider an affine basis $(f_i)_{1\leq i\leq N+1}$ of $\R^N$, and then define the following affine functions
from $\R^N$ to $\R$, with $y=\sum_i y_i f_i$ and ${\sum_i y_i=1}$,
\begin{equation*}
H_i(y)=\\\begin{cases}
  a_1^iy_1+\cdots+a_{n+1}^iy_{n+1} & \text{for } 1\leq i\leq n+1\\
  a_1^iy_1+\cdots+a_{n+1}^iy_{n+1}+y_i & \text{for } n+2\leq i\leq N+1 
\end{cases}\text{.}
\end{equation*}
Then the affine hyperspaces $\{H_i=0\}$ for $1\leq i\leq N+1$ are affinely independant points in the dual space, hence 
$$S_N=\bigl\{ y \mid H_i(y)>0,  1\leq i\leq N+1\bigr\}$$ is an $N$-simplex of $\R^N$.
Now notice that the  intersection of that simplex with the affine space 
$$A_n=\{y \mid y_{n+2}=\cdots=y_{N+1}=0 \}$$  is affinely equivalent to ${\mathcal P}$, using the
map $$(x_1,\ldots,x_{n+1})\mapsto (x_1,\ldots,x_{n+1},0,\ldots,0)\text{.}$$
\end{proof}

\bigskip 

The  implication  $(b) \Rightarrow (a)$ in  Theorem \ref{plongementisometrique} is due to Bruno~Colbois and Patrick~Verovic \cite{cvp2}, who actually proved that if one 
can \textsl{quasi-isometrically} embed a bounded Hilbert geometry $(\Omega,d_\Omega)$ into a finite dimensional normed vector space $(V,||\cdot||)$, 
then the boundary $\partial\Omega$ admits at most a finite number of extremal points.
Let us make a slight variation of  
their proof, assuming an isometric embedding $f\colon(\Omega,d_\Omega)\to (V,||\cdot||)$ is given.  
The proof relies on the following important two facts:
\begin{enumerate}[(i)]
\item The unit sphere of a normed vector space of finite dimension is compact, therefore a 
maximal set of 
$1$-separated points (\textsl{i.e.} a set in which any two distinct points are at distance at least $1$) is finite. Let $N$ be
the cardinality of such a set.
\item \label{limit}  If $x_\infty$ and $y_\infty$ are two extremal points on the boundary  $\partial\Omega$ 
of a Hilbert geometry with supporting hyperplanes not containing the line $(x_\infty y_\infty)$; $o$ a point in  $\Omega$; 
$x(t)$, $y(t)$ two geodesics rays from $o$ to respectively $x_\infty$ and $y_\infty$; then
\begin{equation}\label{eqlimit}
  \lim_{t\to \infty} \dfrac{d_\Omega\bigl(x(t),y(t)\bigr)}{2t}=1\text{.}
\end{equation}
\end{enumerate}

Now let us suppose that the boundary  $\partial\Omega$ admits $N+1$ distinct 
extremal points $x^1_\infty,\ldots,x_\infty^{N+1}$ with suppoting hyperplanes not containing any two of them, 
and let us fix a  point $o$ in $\Omega$ and suppose that the image of $o$ is the origin of $V$.
Let us denote by $x^i(t)$ a geodesic ray from $o$ to $x^i_\infty$.

Then for any positive real number $t\in \R^+_*$ we have on the one hand
\begin{equation}
  \label{eqlimitcons1}
  \left\Vert\frac{f(x^i(t)}{t}\right\Vert=\dfrac{d_\Omega\Bigl(o,f\bigl(x^i(t)\bigr)\Bigr)}{t}=1
\end{equation}
hence $f(x^i(t)/t$ lies on the unit sphere of $(V,||\cdot||)$.
On the other hand, using the formula (\ref{eqlimit}) 
we can find $T\in \R^+$ such that for any $1\leq i<j\leq N+1$  and $t>T$,
\begin{equation}
  \label{eqlimitcons2}
  \left\Vert\frac{f(x^i(t)}{t}-\frac{f(x^j(t)}{t}\right\Vert=\dfrac{d_\Omega\bigl(x^i(t),y^j(t)\bigr)}{t}\geq1\text{.}
\end{equation}

Therefore, for $t>T$, the family $\bigl(f(x^i(t))/t\bigr)$ is a $1$-separated family on the unit sphere,
which admits $N+1$ points. This is in contradiction with the maximality of $N$. 
Hence the boundary $\partial\Omega$ admits 
no more than $N$ such extremal points. 

Now if $\Omega$ is not a polytope, it admits
a two-dimensional section wich is not a polygon, and then by Krein-Millman's Theorem we can
find a sequence of distinct extremal points whose supporting lines 
do not contain any two of them. Hence we cannnot embed it into a normed vector space.

\section{Polytopal Hilbert geometries are bi-Lipschitz to
Euclidean vector spaces}\label{bilipschitz}

% \begin{quote}
%   The barycenter of a polytope and its faces induce a decomposition  of the 
% polytope into pyramids with apex the barycenter and base the faces. These pyramids
% also give rise to cones with summit their apex which in turn decompose
% the ambient space. We built a map which sends these pyramids to their corresponding cones
% and which is a bi-lipschitz map between the Hilbert geometry of the polytope and the
% Euclidean geometry of the ambient space.
% \end{quote}

\begin{theo}
\sl  An $n$-dimensional polytopal  Hilbert geometry  $({\mathcal P},d_{\mathcal P})$ is bi-Lipschitz equivalent to the
$n$-dimensional Euclidean geometry ($\R^n,\Vert\cdot\Vert)$. In other words, there exists a map 
$\Phi\colon{\mathcal P}\to \R^n$ and a constant $A$ such that for any two points $x$ and $y$ in ${\mathcal P}$,
$$
\frac{1}{A}\cdot\Vert \Phi(x) - \Phi(y) \Vert \leq d_{\mathcal P}(x,y)\leq A\cdot\Vert \Phi(x)-\Phi(y)\Vert\text{.}
$$
\end{theo}

This theorem was proved by Bruno~Colbois, Patrick~Verovic and Constantin~Vernicos in dimension $2$ \cite{cvv4}, and
independently by Andreas~Bernig \cite{andreas} and by the author~\cite{ver8} in all dimensions.

Both proofs consist in building a bi-Lipschitz map. A.~Bernig shows that if the convex polytope is
defined by $N$ affine maps $L_1,\ldots,L_N$ as follows 
$$
{\mathcal P}=\{x\mid L_i(x)>0, 1\leq i\leq N\}\text{,}
$$ 
then the map
$$x\mapsto \Phi_b(x)=\sum_{i=1}^N \log L_i(x)\cdot dL_i$$
 is a bi-Lispchitz map onto the dual vector space (notice that the linear part of $L_i$ coincides with $dL_i$). 
This map is easily seen to be Lipschitz continuous. The difficult part in A.~Bernig's proof is to show 
that this map is onto.
 
Our construction is more geometric and the map we build is easily seen to be a bijection. The tricky
part is to prove that it is a bi-Lipschitz map. Our proof recedes through  the following four steps:

\begin{enumerate}[(i)]
\item Using the barycentric subdivision, 
we decompose a polytopal $\cal P$ domain of $\R^n$ 
into a finite number of simplices $S_i$, which
we call \textsl{barycentric simplexes} and which happen to be conical 
flag neighborhoods of the polytope.
\begin{figure}[h]
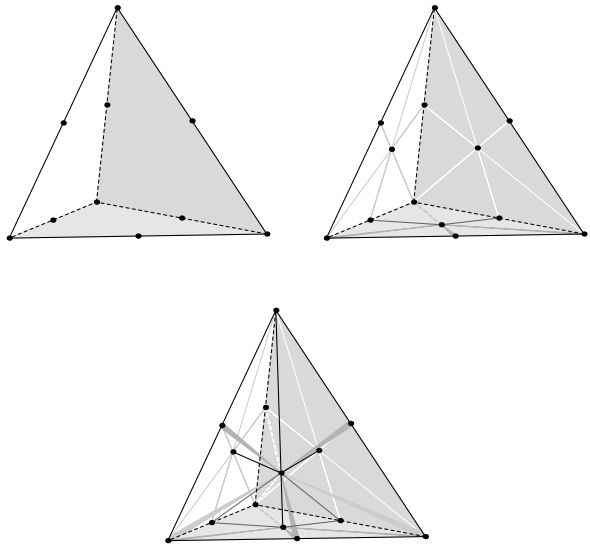

 $$\includegraphics[scale=.4]{polygones.4}\qquad\includegraphics[scale=.4]{polygones.5}
$$

$$
\includegraphics[scale=.4]{polygones.6}
$$

  \caption{The last three steps of the decomposition in dimension $3$}
  \label{figcelldecomp}
\end{figure}

\item  The second steps consists in proving that each simplex $S_i$ admits 
a bi-lipschitz embedding $L_i$
onto a \textsl{fixed} or \textsl{standard} barycentric simplexe of the $n$-simplex. 
The map is a linear one, the difficult part is to prove that it-is bi-lipschitz.
\begin{figure}[h]
  \centering
  \includegraphics[scale=.6]{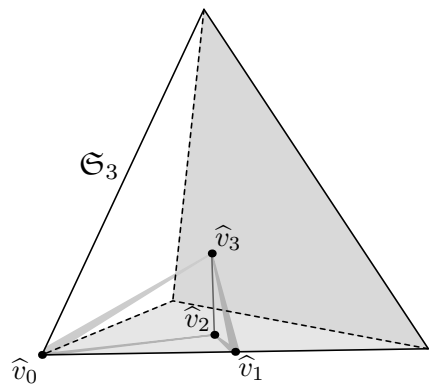}
  \caption{The standard barycentric $3$-simplex of the $3$-simplex}
  \label{figcellsimplex}
\end{figure}

\item  In the third step, 
We show that we can send isometrically the barycentric simplex of an $n$-simplex
onto a cone of a vector space $W_n$, using P.~de~La~Harpe's map $\Phi_n$ 
between the $n$-simplex and $W_n$ . This cone is then sent in a bi-lipschitz way
to the cone associated to a barycentric simplex of a polytope thanks to the inverse of
the map $L_i$ denoted by $M_i$ in the figure \ref{lastfig}.
\begin{figure}[h] 
  \centering
 \includegraphics[scale=.45]{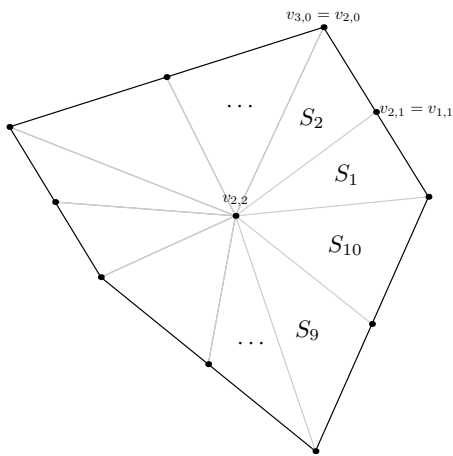}

  \caption{Barycentric simplices of a polygon\label{dintroo}}
\end{figure}

\begin{figure}[h] 
  \centering
  \includegraphics[scale=.5]{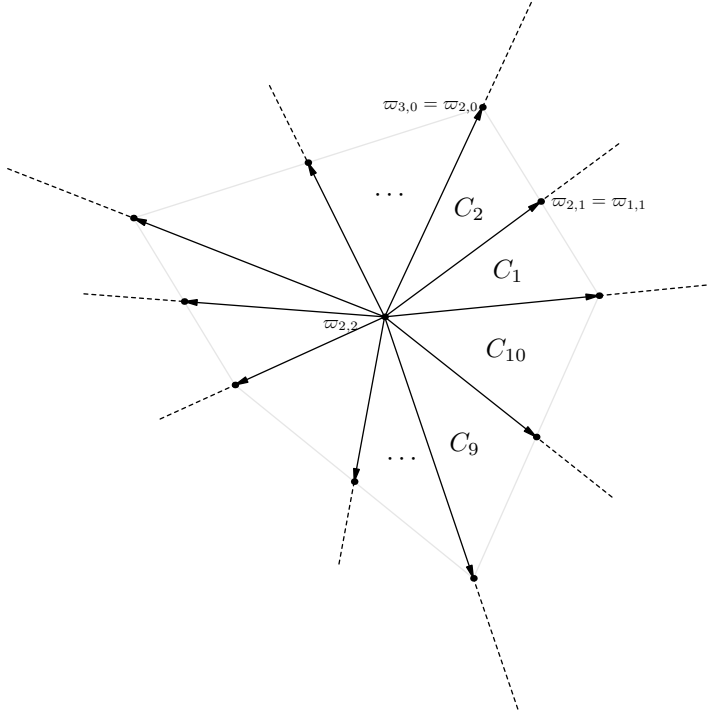}

  \caption{Barycentric cones of a polygon\label{dintrot}}
\end{figure}

\item Finally this allows us to define a map from the polytopal 
domain to $\R^n$ by patching the bi-Lipschitz embeddings associated to each 
of its barycentric simplices.

\begin{figure}[h]
  \centering
  \includegraphics[scale=.45]{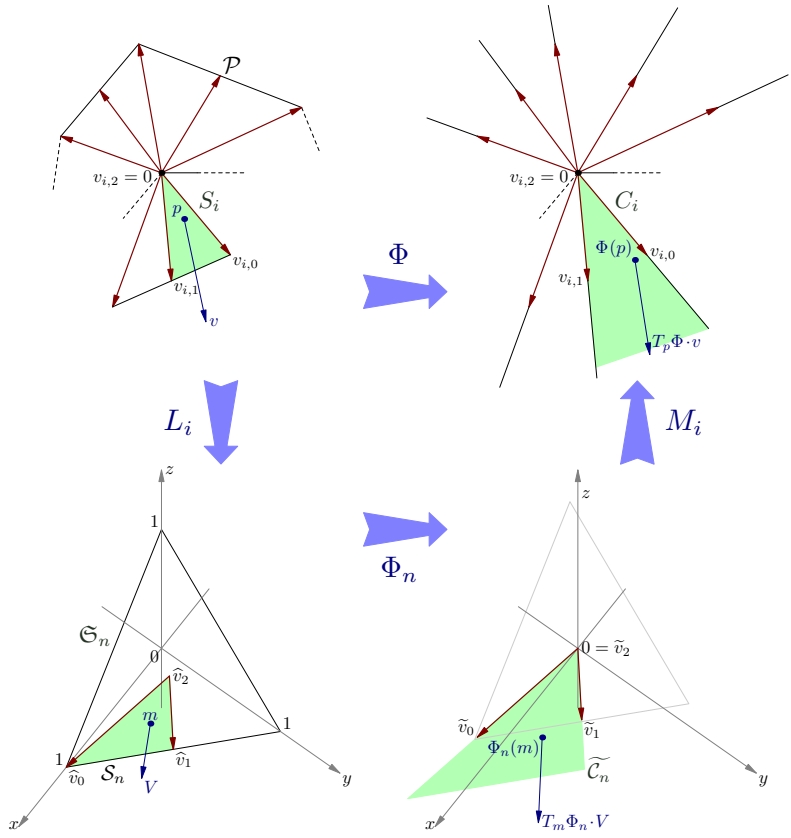}
  \caption{The application $\Phi$ in dimension $2$ illustrated}
\label{lastfig}
\end{figure}
\end{enumerate}

Let us finish by stating the main ingredient of this proof
which is a comparison theorem and which is interesting on its own:

 As in  formula (\ref{eqhilbertisfinsler}) we denote by $F_\mathcal{C}$ the finsler metric associated
 to the convex set $\mathcal{C}$.

 \begin{theo}\label{mainprop}
 \sl  Let $\mathcal{A}$ and $ \mathcal{B}$ be two convex sets with a common conical flag neighborhood $\mathcal{S}$.
 There exists a constant $C$ such that for any $x\in \mathcal{S}$ and $v\in\R^n$ one has
 \begin{equation}
   \frac{1}{C}\cdot F_\mathcal{B}(x,v)\leq F_\mathcal{A}(x,v)\leq C\cdot F_\mathcal{B}(x,v)\text{.}
 \end{equation}
 \end{theo}

 \begin{ex}\label{exeight}
 In the two-dimensional case the condition is that $\mathcal{A}$ and 
$\mathcal{B}$ contain
 a triangle $\mathcal{S}$ with one of its edges on their boundaries, 
a unique vertex of which is an extremal point of both of them where they fail to be $C^1$ (see figure \ref{figexeight}).
 \begin{figure}[H]
   \centering

   \includegraphics[scale=.5]{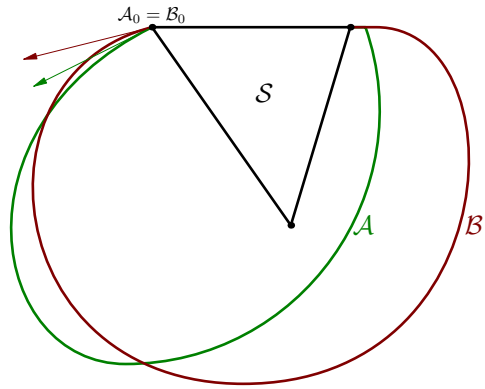}
   \caption{Illustration of Example \ref{exeight}}
   \label{figexeight}
 \end{figure} 
 \end{ex}

\def\cprime{$'$}
\providecommand{\bysame}{\leavevmode\hbox to3em{\hrulefill}\thinspace}
\providecommand{\MR}{\relax\ifhmode\unskip\space\fi MR }
\providecommand{\MRhref}[2]{%
  \href{http://www.ams.org/mathscinet-getitem?mr=#1}{#2}
}
\providecommand{\href}[2]{#2}

\end{document}